\documentclass[12pt, reqno]{amsart}
\usepackage[all]{xy}

\usepackage{amssymb}
\usepackage{xcolor}
\usepackage{amsthm}
\usepackage{amsmath,cancel}
\usepackage{amscd}
\usepackage{verbatim}
\usepackage{float}
\usepackage{color,mdframed}
\usepackage{bbm}
\usepackage[mathscr]{eucal}
\usepackage[all]{xy}
\usepackage{bm}

\usepackage[OT2,T1]{fontenc}

\usepackage[draft]{hyperref}
\usepackage{mathtools}

\usepackage{comment}
\usepackage{calligra}
\usepackage{stmaryrd,float}
\DeclareMathAlphabet{\mathcalligra}{T1}{calligra}{m}{n}
\usepackage{enumitem}
\setitemize{leftmargin=*}
\setlist[enumerate]{leftmargin=*,label=\rm{(\arabic*)}}

\theoremstyle{plain}
\newtheorem{theorem}{Theorem}
\newtheorem{proposition}[theorem]{Proposition}

\newtheorem{lemma}[theorem]{Lemma}
\newtheorem*{theorem*}{Theorem}

\theoremstyle{definition}

\theoremstyle{remark}

\newcommand{\R}{\mathbb{R}}
\newcommand{\Q}{\mathbb{Q}}
\newcommand{\Z}{\mathbb{Z}}
\newcommand{\N}{\mathbb{N}}
\def\C{\mathbb{C}}

\renewcommand{\H}{\mathbb{H}}

\newcommand{\z}{\mathfrak{z}}

\newcommand{\zxz}[4]{\begin{pmatrix} #1 & #2 \\ #3 & #4 \end{pmatrix}}
\newcommand{\abcd}{\zxz{a}{b}{c}{d}}
\newcommand{\leg}[2]{\left( \frac{#1}{#2} \right)}

\newcommand{\calF}{\mathcal{F}}

\newcommand{\calH}{\mathcal{H}}

\newcommand{\calL}{\mathcal{L}}

\renewcommand{\a}{\alpha}

\newcommand{\g}{\gamma}
\renewcommand{\l}{\lambda}
\renewcommand{\th}{\theta}
\newcommand{\vth}{\vartheta}

\newcommand{\eps}{\varepsilon}

\newcommand{\sm}{\setminus}

\renewcommand{\t}{\tau}

\newcommand{\re}{\operatorname{Re}}

\newcommand{\flo}[1]{\lfloor #1\rfloor}
\newcommand{\pflo}[1]{\left\lfloor #1\right\rfloor}

\newcommand{\SL}{{\text {\rm SL}}}

\newcommand{\G}{\Gamma}

\newcommand{\Li}{\operatorname{Li}}

\newif\ifdefs

\defstrue

\newif\ifdiscs
\discstrue

\numberwithin{equation}{section}
\numberwithin{theorem}{section}

\allowdisplaybreaks

\ifdefs
\else
\excludecomment{extradetails}
\excludecomment{extradetailsproof}
\fi

\ifdiscs
\else
\excludecomment{discussion}
\fi

\title{Generalized $L$-functions related to the Riemann zeta function}
\author{Kathrin Bringmann}
\author{Ben Kane}
\author{Srimathi Varadharajan}
\address{University of Cologne, Department of Mathematics and Computer Science, Weyertal 86-90, 50931 Cologne, Germany}
\email{kbringma@math.uni-koeln.de}
\address{Department of Mathematics, University of Hong Kong, Pokfulam, Hong Kong}
\email{bkane@hku.hk}
\address{Cell Biology and Anatomy Department, University of Calgary, 3330 Hospital Drive NW, Calgary, Alberta, CANADA T2N 1N4}
\email{srimathi.Varadharaja@ucalgary.ca}
\date{\today}

\begin{document}

\begin{abstract}
	In this paper, we construct generalized $L$-functions associated to meromorphic modular forms of weight $\frac12$ for the theta group with a single simple pole in the fundamental domain. We then consider their behaviour towards $i\infty$ and relate this to the Riemann zeta function.
\end{abstract}
\maketitle

\section{Introduction and statement of results}

Arithmetic information $c(n)$ for $n\in\N$ can be naturally encoded in a so-called Dirichlet series, defined for $\re(s)$ sufficiently large, 
\[
	L(s)=\sum_{n=1}^{\infty} \frac{c(n)}{n^s}.
\]
One calls the function $L$ a (classical) {\it $L$-function} if it satisfies the following three properties:
\begin{enumerate}
	\item It has a meromorphic continuation to the whole complex plane.
	
	\item It has an {\it Euler product}, for $\re(s)$ sufficiently large,
	\[
		L(s)=\prod_{p} \frac{1}{1-f_p\left(p^{-s}\right)p^{-s}}.
	\]
	where the product runs over all primes and $f_p$ is a polynomial.
	
	\item There is some archimedean information $L_{\infty}(s)$ such that the function 
	\[
		\Lambda(s):=L_{\infty}(s)L(s)
	\]
	satisfies, for some $k$ and $\varepsilon=\pm 1$, a {\it functional equation}
	\[
		\Lambda(k-s)=\varepsilon \Lambda(s).
	\]
\end{enumerate}
In \cite{BKgenL}, the first two authors constructed functions $L_z$ via regularized Mellin transforms of weight two meromorphic modular forms on $\SL_2(\Z)$ with a single simple pole at one point $z$ in the fundamental domain (not on the imaginary axis). These functions satisfy a functional equation and their behaviour as $z\to i\infty$ was shown to be related to a classical $L$-function (see \cite[Theorem 1.1 and Theorem 4.3]{BKgenL}), so they were named generalized $L$-functions.

Slightly more formally, for $k\in\frac12\Z$ we call a collection of functions $s\mapsto\calL_z(s)$ defined for $s\in\C$ for almost all $z\in\H$ {\it generalized $L$-functions of weight $k$} if they satisfy the following properties:
\begin{enumerate}
	\item We have the functional equations
	\[
		\calL_z(k-s) = \pm\calL_z(s).
	\]
	
	\item There exist ``simple'' functions $C_{\ell,s}$ and $D_{\ell,s}$ such that, for almost all $x\in[0,1]$,
	\begin{equation}\label{eqn:limClassical}
		\lim_{y\to \infty} \left(\calL_{z}(s) - \sum_{\ell\ge0} C_{\ell,s}(x) y^{s-\ell} - \sum_{\ell\ge0} D_{\ell,s}(x) y^{k-s-\ell}\right)
	\end{equation}
	is a classical completed $L$-function.
\end{enumerate}
We call the classical $L$-function whose completion one obtains from the limit \eqref{eqn:limClassical}, the {\it$L$-function associated to $\calL_z(s)$.}

Letting $j_\vth$ be the Hauptmodul for $\G_\vth$, explicitly given in \eqref{eqn:jthetadef}, we (formally) define
\begin{equation}\label{eqn:Fzdef}
	F_z(s) := s\left(\frac12-s\right)j_\vth(z)\int_0^\infty \frac{\vth(it)}{j_\vth(it)-j_\vth(z)}t^{s-1} dt.
\end{equation}
The main result in this paper is the fact that the functions $F_z$ give a collection of generalized $L$-functions related to the Riemann zeta function in the sense that \eqref{eqn:limClassical} is the Riemann $\xi$-function
\[
	\xi(s) := \frac{s(s-1)}{2\pi^\frac s2}\G\left(\frac s2\right)\zeta(s),
\]
which satisfies the functional equation
\begin{equation}\label{eqn:xifunctional}
	\xi(1-s)=\xi(s).
\end{equation}
Here $\Gamma(s)$ is the $\Gamma$-function, defined by $\Gamma(s):=\int_{0}^{\infty}t^{s-1}e^{-t}dt$ for $\re(s)>1$ and extended meromorphically to the entire complex plane.

\begin{theorem}\label{thm:main}
	The functions $F_z$ are a collection of generalized $L$-functions related to the Riemann zeta function in the following sense:
	\begin{enumerate}[leftmargin=*,label=\rm{(\arabic*)}]
		\item If $z$ is not $\Gamma_\vth$-equivalent to a point on $i\R^+$, then the integral in \eqref{eqn:Fzdef} converges for all $s\in\C$.
		
		\item If $z$ is not $\Gamma_\vth$-equivalent to a point on $i\R^+$, then we have 
		\[
			F_z\left(\frac{1}{2}-s\right)=F_{z}(s).
		\]
		
		\item With $C_{\ell,s}(x)$ and $D_{\ell,s}(x)$ defined in \eqref{eqn:Cells} and \eqref{eqn:Dells}, respectively, the limit \eqref{eqn:limClassical} is $\xi(2s)$.
	\end{enumerate}
\end{theorem}

The paper is organized as follows. In Section \ref{sec:prelim}, we recall the properties of the theta function and the zeta function. In Section \ref{sec:meroconstruct}, we construct 
the Hauptmodul $j_\vth$ for $\G_\vth$. In Section \ref{sec:MellinTransform}, we show that $F_z$ converges absolutely for all $s\in\C$ under a mild assumption on $z$ and then finally, for $z=x+iy$ with $x\notin\Z$ fixed, take the limit $y\to\infty$ to prove Theorem \ref{thm:main}.

\section*{Acknowledgements}

The first author is supported by the Deutsche Forschungsgemeinschaft (DFG) Grant No. BR 4082/5-1. The research of the second author was supported by grants from the Research Grants Council of the Hong Kong SAR, China (project numbers HKU 17303618 and 17314122). Part of the research was conducted while the third author was a postdoctoral fellow at the University of Hong Kong.

\section{Preliminaries}\label{sec:prelim}

\subsection{The theta function}

Define the {\it theta function} ($q:=e^{2\pi i\t}$)
\[
	\vth(\t):=\sum_{n\in\Z}q^{\frac{n^2}{2}}.
\]
It is well-known that $\vth$ is a modular form of weight $\frac12$ for the {\it theta group}
\[
	\G_\vth := \left\langle S,T^2\right\rangle = \left\{\abcd\in\SL_2(\Z) : a\equiv d\pmod{2}\text{ and }c\equiv b\pmod{2}\right\},
\]
where $S:=\left(\begin{smallmatrix}0&-1\\ 1&0\end{smallmatrix}\right)$ and $T:=\left(\begin{smallmatrix}1&1\\ 0&1\end{smallmatrix}\right)$. The theta function has the following growth behavior.

\begin{lemma}\label{lem:ThetaGrowth} Let $\tau=u+iv$.
	\ \begin{enumerate}[leftmargin=*,label=\rm{(\arabic*)}]
		\item As $v\to\infty$, we have 
		\[
			\vartheta(\t) = 1+O\left(e^{-\pi v}\right).
		\]
		
		\item As $v\to 0^+$, we have 
		\[
			\vartheta(iv) = \frac{1}{\sqrt{v}} + O\left(\frac{e^{-\frac{\pi}{v}}}{\sqrt{v}}\right).
		\]
	\end{enumerate}
\end{lemma}

\begin{proof}\leavevmode\newline
	(1) The claim follows directly from the Fourier expansion of $\vth$.\\
	(2) It is well-known that we have
	\begin{equation}\label{eqn:varthetamodular}
		\vartheta(\t)=(-i\t)^{-\frac12}\vartheta\left(-\frac{1}{\t}\right).
	\end{equation}
	Taking $\t=it$ and plugging in the Fourier expansion for $\vartheta(\frac{i}{t})$ yields the claim.
\end{proof}

\subsection{A regularized Mellin transform of the theta function}

We recall the following well-known relation between the theta function and the Riemann zeta function that goes back to Riemann.\footnote{See \cite[Subsection 2.6]{Titchmarsch} for a modern reference.}

\begin{lemma}\label{lem:thetazeta}
	For any $t_0>0$, we have
	\[
		\frac{2}{s(2s-1)}\xi(2s) = \int_{t_0}^\infty (\vth(it)-1)t^{s-1} dt + \int_0^{t_0} \left(\vth(it)-\frac{1}{\sqrt{t}}\right)t^{s-1} dt - \frac{t_0^s}{s} + \frac{t_0^{s-\frac12}}{s-\frac12}.
	\]
\end{lemma}

\subsection{Asymptotics for special functions}

For $y>0$ and $s\in\C$ we define the {\it incomplete gamma function} by
\[
	\Gamma(s,y):=\int_y^{\infty} t^{s-1}e^{-t}dt.
\]
Denoting, for $\ell\in\N$, the {\it rising factorial} by $(a)_\ell:=\prod_{j=0}^{\ell-1}(a+j)$, we require the following asymptotic behavior for $\Gamma(s,y)$, which may be found in \cite[8.11.2]{NIST}.

\begin{lemma}\label{lem:incGammaBound}
	For $s\in\C$ and $N\in\N$ we have, as $y\to\infty$,
	\[
		\Gamma(s,y)= y^{s-1}e^{-y}\left(\sum_{j=0}^{N-1}\frac{(-1)^j(1-s)_{j}}{y^j}+O\left(y^{-N}\right)\right).
	\]
\end{lemma}

We also require bounds for the confluent hypergeometric function ${_1F_1}$. Assuming that $s\notin\Z$, \cite[13.7.1]{NIST} implies that, as $y\to\infty$, for any $N\in\N_0$
\begin{equation*}
	{_1F_1}(s;s+1;y) \sim \frac{se^y}{y}\left(\sum_{j=0}^N \frac{(1-s)_j}{y^j}+O_{s,N}\left(y^{-N-1}\right)\right).
\end{equation*}

\section{Construction of meromorphic modular forms}\label{sec:meroconstruct}

In this section, we construct meromorphic modular forms $\calH_z$ whose regularized Mellin transforms give the collection of generalized $L$-functions that are used to prove Theorem \ref{thm:main}.

\subsection{The Hauptmodul $j_\vth$ for the theta group}

In this subsection, we construct a Hauptmodul for the theta group $\G_\vth$ and discuss its properties. We recall that a {\it Hauptmodul} for a congruence subgroup $\G\subseteq\SL_2(\Z)$ is a $\G$-invariant meromorphic function $j_\G$ for which every $\G$-invariant meromorphic function may be written as a rational function in $j_\G$. For this, we require the modular {\it lambda function}
\[
	\lambda(\tau):=\frac{\theta_2(\tau)^4}{\vth(\tau)^4},
\]
where 
\[
	\th_2(\t) := \sum_{n\in\Z} e^{\pi i\left(n+\frac12\right)^2\t}.
\]
It is well-known that $\l$ is a Hauptmodul for $\G(2)$ and satisfies the identities 
\begin{equation}\label{eqn:lambdainversion}
	\l\left(-\frac1\t\right) = 1 - \l(\t),\qquad \l(\t+1) = \frac{\l(\t)}{\l(\t)-1},\qquad \l\leg{1}{1-\t} = \frac{1}{1-\l(\t)}.
\end{equation}
We have the following growth towards the cusps.\footnote{There are three cusps of $\G(2)$, namely, $0$, $1$, and $i\infty$, each with cusp width $2$.} For this, let $\H^*:=\H\cup\Q\cup\{i\infty\}$ and $\C^*:=\C\cup\{i\infty\}$.

\begin{lemma}\label{lem:lambdagrowth}
	\ \begin{enumerate}
		\item As $v\to\infty$, we have 
		\begin{equation*}
			\l(\t) = 16e^{\pi i\t} + O\left(e^{-2\pi v}\right),\qquad \frac{1}{\l(\t)} = \frac{e^{-\pi i\t}}{16} + O(1).
		\end{equation*}
	
		\item As $v\to\infty$, we have
		\begin{equation*}
			\l\left(-\frac1\t\right) = 1 - 16e^{\pi i\t} + O\left(e^{-2\pi v}\right),\qquad \frac{1}{\l\left(-\frac1\t\right)-1} = -\frac{e^{-\pi i\t}}{16} + O(1).
		\end{equation*}

		\item As $v\to\infty$, we have
		\begin{equation*}
			\lambda\left(\frac{\tau}{\tau+1}\right)=\frac{e^{-\pi i\tau}}{16} +O(1).
		\end{equation*}
		
		\item The function $\l$ is a bijection from $\G(2)\sm\H^*$ to $\C^*$.
	\end{enumerate}
\end{lemma}

\begin{proof}\leavevmode\newline
	(1) The statements follow immediately from the Fourier expansion of $\l$.\\
	(2) The first claim follows from \eqref{eqn:lambdainversion} and part (1) and the second follows by subtracting $1$ from both sides and dividing.\\
	(3) Similarly, we combine the first and second equations of \eqref{eqn:lambdainversion} to obtain
	\begin{equation}\label{eqn:1/lambdatau}
		\lambda\left(\frac{\tau}{\tau+1}\right)=\frac{1}{\lambda(\tau)}.
	\end{equation}
	The claim now follows immediately from part (1).\\
	(4) This follows since a $\G$-invariant meromorphic function is a Hauptmodul for $\G$ if and only if it is a bijection from $\G\sm\H^*$ to $\C^*$.
\end{proof}

Define
\begin{equation}\label{eqn:jthetadef}
	j_\vth(\tau):=\frac{1}{\lambda(\tau)\left(1-\lambda(\tau)\right)}.
\end{equation}
We now show that $j_\vth$ is a Hauptmodul for $\Gamma_\vth$, and give its growth towards the cusps.

\begin{lemma}\label{lem:jthetagrowth}
	\ \begin{enumerate}
		\item The function $j_\vth$ is $\Gamma_\vth$-invariant.
		
		\item As $v\to\infty$, we have
		\[
			j_\vth\left(-\frac{1}{\tau}\right)=j_\vth(\tau)=\frac{e^{-\pi i \tau}}{16} +O(1).
		\]
		In particular, as $v\to 0^+$,
		\[
			j_\vth(iv)= e^{\frac{\pi}{v}} +O(1).
		\]
		
		\item As $v\to\infty$, we have
		\[
			j_\vth\left(\frac{\tau}{\tau+1}\right)=-256e^{2\pi i \tau}+O\left(e^{-3\pi v}\right).
		\]
		
		\item The function $j_\vth$ is a bijection from $\G_\vth\sm\H^*$ to $\C^*$. Every meromorphic modular function on $\G_\vth\sm\H$ can be written as a rational function in $j_\vth$.
	\end{enumerate}
\end{lemma}

\begin{proof}\leavevmode\newline
	(1) It is enough to show that $	j_\vth$ is invariant under the generators of $\G_\th$, $S$, and $T^2$, which follows from the first equation in \eqref{eqn:lambdainversion} by a direct calculation.\\
	(2) The first identity follows immediately from (1). For the second identity, we first note that, by \eqref{eqn:lambdainversion},
	\begin{equation}\label{eqn:1/1-lambdatau}
		\frac{1}{1-\lambda(\tau)}= 1+O\left(e^{\pi i \tau}\right).
	\end{equation}
	Combining \eqref{eqn:1/1-lambdatau} with \eqref{eqn:1/lambdatau} and Lemma \ref{lem:lambdagrowth} (3), we conclude that, as $v\to\infty$,
	\[
		j_\vth(\t) = \frac{1}{\l(\t)}\frac{1}{1-\l(\t)} = \left(\frac{e^{-\pi i\t}}{16}+O(1)\right)\left(1+O\left(e^{\pi i\t}\right)\right) = \frac{e^{-\pi i\t}}{16} + O(1).
	\]
	(3) A short calculation using \eqref{eqn:1/lambdatau} and \eqref{eqn:1/1-lambdatau} shows that
	\[
		j_\vth\left(\frac{\t}{\t+1}\right) = -256e^{2\pi i\t} + O\left(e^{-3\pi v}\right).
	\]
	(4) To show surjectivity, let $c\in\C^*\sm\{0\}$. Then we have 
	\[
		j_\vth(\t) = c \Leftrightarrow \l(\t)^2 - \l(\t) + \frac1c = 0.
	\]
	By Lemma \ref{lem:lambdagrowth} (4), for any root $\a$ of the polynomial $x^2-x+\frac1c$ there exists $\t\in\H$ such that $\l(\t)=\a$. For $c=0$, letting $\t\to i\infty$ in part (3) implies that $j_\vth(1)=0$. So $j_\vth$ is surjective.

	To show that $j_\vth$ is injective, suppose for contradiction that $j_\vth(\t_1)=j_\vth(\t_2)$ with $\t_1$ and $\t_2$ not $\G_\vth$-equivalent. Then 
	\[
		\lambda\left(\tau_1\right)\left(1-\lambda(\tau_1)\right)=\lambda(\tau_2)\left(1-\lambda(\tau_2)\right).
	\]
	Setting $\a_1:=\lambda(\tau_1)$ and $\a_2:=\lambda(\tau_2)$ and rearranging, we have 
	\begin{equation}\label{eqn:quadratica2a1}
		\a_2^2-\a_2 +\a_1-\a_1^2=0.
	\end{equation}
	Since we have a quadratic equation in $\a_2$, this has exactly two solutions (counting multiplicity) in $\C$. One directly checks that $\a_2=\a_1$ and $\a_2=1-\a_1$ are both solutions to \eqref{eqn:quadratica2a1}.
	If $\a_1\neq \frac{1}{2}$, then these are distinct solutions, while for $\a_1=\frac{1}{2}$ we see that \eqref{eqn:quadratica2a1} becomes $\a_2^2-\a_2+\frac{1}{4}=0$, in which case $\a_2=\frac{1}{2}=\a_1=1-\a_1$ is a double root.
	We see that in both cases the two solutions (counting multiplicity) to \eqref{eqn:quadratica2a1} are $\a_2=\a_1$ and $\a_2=1-\a_1$.
	
	If $\a_2=\a_1$, then by Lemma \ref{lem:lambdagrowth} (4) we conclude that $\t_2$ is $\G(2)$-equivalent to $\t_1$. Since $\G(2)\subseteq\G_\vth$, this contradicts the assumption that $\t_1$ and $\t_2$ are not $\G_\vth$-equivalent. Hence $\a_2=1-\a_1$. From the first equation of \eqref{eqn:lambdainversion}, we conclude that $\l(\t_2)=\l(-\frac{1}{\t_1})$.
	By Lemma \ref{lem:lambdagrowth} (4), there exists $\gamma\in \G(2)$ such that 
	\[
		\t_2 = \g\frac{-1}{\t_1} = \g \circ S\t_1. 
	\]
	Since $\g\in\G(2)\subset\G_\vth$ and $S\in\G_\vth$, we have $\g\circ S\in\G_\vth$, which contradicts the assumption that $\t_1$ and $\t_2$ are not $\G_\vth$-equivalent. We therefore conclude that $j_\vth$ is a $\G_\vth$-invariant meromorphic function and is a bijection from $\G_\vth\sm\H^*$ to $\C^*$. Hence it is a Hauptmodul for $\G_\vth$ by the equivalence noted in the proof of Lemma \ref{lem:lambdagrowth} (4).
\end{proof}

\subsection{Definition of the meromorphic modular forms}\label{sec:Hzdef} 

For $z\in\H$, define
\[
	\calH_{z}(\tau):=\frac{j_\vth(z)\vartheta(\tau)}{j_\vth(\tau)-j_\vth(z)}.
\]
A direct calculation yields the following properties of $\calH_z$.

\begin{lemma}\label{lem:calHprop}
	\ \begin{enumerate}
		\item The function $\calH_z$ is modular for $\G_\vth$ of weight $\frac12$. Moreover, it has a pole at $\tau=\mathfrak{\z}$ if and only if $\mathfrak{\z}$ is $\Gamma_\vth$-equivalent to $z$.
		
		\item The function $z\mapsto\calH_z(\t)$ is $\G_\vth$-invariant.
		
		\item For $\tau\in \H$ fixed, we have 
		\[
			\lim_{z\to i\infty} \calH_{z}(\tau) =\vartheta(\tau).
		\]
	\end{enumerate}
\end{lemma}

\section{Regularized Mellin transforms and the proof of Theorem \ref{thm:main}}\label{sec:MellinTransform}

In this section, we investigate the properties of $F_z$ defined in \eqref{eqn:Fzdef} and prove Theorem \ref{thm:main}.

\subsection{Convergence and functional equation}

We first show that,
 under
\rm
 a mild assumption on $z$,
the function $F_z$ 
\rm
converges for all $s\in\C$.

\begin{proposition}\label{prop:exist}
	The function $F_{z}$ is well-defined for $z\in\H$ that is not $\G_\vth$-equivalent to any point on $i\R^+$. Moreover, we have 
	\begin{equation}\label{E:func}
		F_{z}\left(\frac{1}{2}-s\right)=F_z(s).
	\end{equation}
\end{proposition}

\begin{proof}
	We first show that $F_z$ is well-defined. Since $j_\vth$ is $\G_\vth$-invariant, we may assume that $z$ lies in the fundamental domain $\calF_\vth:=\calF\cup T^{-1}\calF\cup(-T^{-1}S\calF)$	for $\G_\vth$, where $\calF$ is the standard fundamental domain for $\SL_2(\Z)$.
	Since $z\neq it$ by assumption, Lemma \ref{lem:jthetagrowth} (4) moreover implies that
	\[
		j_\vth(it) - j_\vth(z) \ne 0.
	\]
	Hence the integrand is finite for every $t>0$ and therefore, for any $0<t_1<t_2<\infty$, the integral
	\[
		\int_{t_1}^{t_2} \frac{\vartheta(it)}{j_\vth(it)-j_\vth(z)} t^{s-1} dt
	\]
	is well-defined and finite. Taking $t_1=\frac{1}{2(y+y^{-1})}$ and $t_2=2(y+y^{-1})$, it remains to show that
	\[
		\int_{0}^{\frac{1}{2\left(y+y^{-1}\right)}} \frac{\vartheta(it)}{j_\vth(it)-j_\vth(z)} t^{s-1} dt<\infty\qquad \text{ and }\qquad \int_{2\left(y+y^{-1}\right)}^{\infty} \frac{\vartheta(it)}{j_\vth(it)-j_\vth(z)} t^{s-1} dt<\infty.
	\]
	For $0<t<\frac{1}{2(y+y^{-1})}$, Lemma \ref{lem:jthetagrowth} (2) implies that
	\[
		\left|\frac{1}{j_\vth(it)-j_\vth(z)}\right| \ll e^{-\frac\pi t}. 
	\]
	Combining this with Lemma \ref{lem:ThetaGrowth} (2) hence yields, that for $t\to 0^+$,
	\[
		\left|\frac{\vartheta(it)}{j_\vth(it)-j_\vth(z)} t^{s-1}\right|\ll t^{\sigma-\frac{3}{2}} e^{-\frac{\pi}{t}},
	\]
	where here and throughout $\sigma:=\re(s)$. Therefore
	\[
		\int_{0}^{\frac{1}{2\left(y+y^{-1}\right)}} \left|\frac{\vartheta(it)}{j_\vth(it)-j_\vth(z)} t^{s-1}\right|dt\ll \int_{0}^{\frac{1}{2\left(y+y^{-1}\right)}}t^{\sigma-\frac{3}{2}} e^{-\frac{\pi}{t}}dt<\infty.
	\]

	Similarly, for $t>2(y+y^{-1})$, Lemma \ref{lem:jthetagrowth} (2) implies that, as $t\to\infty$,
	\[
		\left|\frac{1}{j_\vth(it)-j_\vth(z)}\right|\ll \frac{1}{\left|e^{\pi t} - e^{\pi\left(y+y^{-1}\right)}\right| +O(1)}\ll \frac{1}{e^{\pi t}-e^{\frac{\pi t}{2}}} \ll e^{-\pi t}.
	\]
	Combining this with Lemma \ref{lem:ThetaGrowth} (1) hence yields that for $t\to\infty$
	\[
		\left|\frac{\vartheta(it)}{j_\vth(it)-j_\vth(z)} t^{s-1}\right|\ll t^{\sigma-1} e^{-\pi t}.
	\]
	Therefore 
	\[
		\int_{2y}^{\infty}\left|\frac{\vartheta(it)}{j_\vth(it)-j_\vth(z)} t^{s-1}\right|dt\ll \int_{2y}^{\infty}t^{\sigma-\frac{3}{2}} e^{- \pi t}dt<\infty.
	\]

	The functional equation \eqref{E:func} now follows by the change of variables $t\mapsto\frac1t$, \eqref{eqn:varthetamodular}, and the invariance of $j_\vth$ under inversion.
\end{proof}

\subsection{Proof of Theorem \ref{thm:main}}

To show Theorem \ref{thm:main}, we require the following.

\begin{proposition}\label{prop:limit}
	For every $x\notin\Z$, we have, with $z=x+iy$,
	\begin{multline*}
		\lim_{y\to\infty} \left({\vphantom{\sum_{\ell=1}^{\left\lfloor\re\left(\frac12-s\right)\right\rfloor}}}F_{z}(s) + \left(\frac{1}{2}-s\right)y^s +sy^{\frac12-s} - s\left(\frac{1}{2}-s\right)\sum_{\ell=1}^{\lfloor\sigma\rfloor} \frac{(1-s)_{\ell-1}}{\pi^\ell}\Li_\ell\left(e^{\pi ix}\right)y^{s-\ell}\right.
		\\
		+ s\left(\frac{1}{2}-s\right)\sum_{\ell=1}^{\lfloor\sigma\rfloor} 
		\frac{(s+1-\ell)_{\ell-1}}{\pi^\ell}\Li_\ell\left(e^{-\pi ix}\right)y^{s-\ell}\\
		- s\left(\frac{1}{2}-s\right)\sum_{\ell=1}^{\left\lfloor\frac12-\sigma\right\rfloor} \frac{\left(s+\frac12\right)_{\ell-1}}{\pi^\ell}\Li_\ell\left(e^{\pi ix}\right)y^{\frac12-s-\ell}
		\\
		\left.+ s\left(\frac{1}{2}-s\right)\sum_{\ell=1}^{\left\lfloor\frac12-\sigma\right\rfloor} \frac{\left(\frac32-s-\ell\right)_{\ell-1}}{\pi^\ell}\Li_\ell\left(e^{-\pi ix}\right)y^{\frac12-s-\ell}\right) = \xi(2s).
	\end{multline*}
\end{proposition}

\begin{proof}
	We split, for some $t_0>0$,
	\[
		F_z(s) = s\left(\frac{1}{2}-s\right) \left(F_{z,0,\frac1y}(s) + F_{z,\frac1y,t_0}(s) + F_{z,t_0,y}(s) + F_{z,y,\infty}(s)\right),
	\]
	where for $0\le y_1\le y_2\le\infty$ we define
	\[
		F_{z,y_1,y_2}(s) := j_\vth(z)\int_{y_1}^{y_2} \frac{\vth(\t)}{j_\vth(\t)-j_\vth(z)}t^{s-1} dt.
	\]
	Since both sides of the claim are invariant under $s\mapsto\frac12-s$ by Proposition \ref{prop:exist} and \eqref{eqn:xifunctional}, we may assume without loss of generality that $\sigma\ge\frac14$.
	
	We claim that in this case, we have
	\begin{align}\label{eqn:Fz01/yfinal}
		F_{z,0,\frac1y}(s) &= \sum_{\ell=1}^{\pflo{\frac12-\sigma}} \frac{\left(\frac32-s-\ell\right)_{\ell-1}}{\pi^\ell}\Li_\ell\left(e^{-\pi ix}\right)y^{\frac12-s-\ell} + o_{x,s}(1),\\
		\label{eqn:Fz1/yt0final}
		F_{z,\frac1y,t_0}(s) &= -\int_\frac1y^{t_0} \left(\vth(it)-\frac{1}{\sqrt t}\right)t^{s-1} dt + \frac{y^{\frac12-s}}{s-\frac12} - \frac{t_0^{s-\frac12}}{s-\frac12} - \frac{\Li_1\left(e^{\pi ix}\right)}{\pi y^{s+\frac12}} \hspace{-.05cm}+\hspace{-.05cm} o_{x,s}(1),\\
		\nonumber
		F_{z,t_0,y}(s) &= -\int_{t_0}^y (\vth(it)-1)t^{s-1} dt + \frac{t_0^{s}}{s} - \frac{y^{s}}{s}\\
		\nonumber
		&\hspace{5cm}- \sum_{\ell=1}^{\flo\sigma} \frac{(1-s)_{\ell-1}}{\pi^\ell}\Li_\ell\left(e^{\pi ix}\right)y^{s-\ell} + o_{x,s}(1),\\
		\label{eqn:Fzyinftyfinal}
		F_{z,y,\infty}(s) &= \sum_{\ell=1}^{\flo\sigma} \frac{(s+1-\ell)_{\ell-1}}{\pi^\ell} \Li_\ell\left(e^{-\pi ix}\right)y^{s-\ell} + o_{x,s}(1).
	\end{align}

	The proofs of \eqref{eqn:Fz01/yfinal}--\eqref{eqn:Fzyinftyfinal} are all similar and also analogous to the proofs of \cite[(4.5)--(4.8)]{BKgenL}. We hence only show \eqref{eqn:Fz1/yt0final} (we choose the case \eqref{eqn:Fz1/yt0final} instead of \eqref{eqn:Fz01/yfinal} to demonstrate how the Mellin transform of $\vth$ naturally appears), leaving the other cases to the interested reader. We split
	\begin{align}\label{eqn1}
		F_{z,\frac1y,t_0}(s) &= \int_\frac1y^{t_0} \frac{j_\vth(z)\left(\vth(it)-\frac{1}{\sqrt t}\right)}{j_\vth(it)-j_\vth(z)}t^{s-1} dt + \int_\frac1y^{t_0} \frac{j_\vth(z)}{j_\vth(it)-j_\vth(z)}t^{s-\frac32} dt,\\
		\nonumber
		F_{z,t_0,y}(s) &= \int_{t_0}^y \frac{j_\vth(z)(\vth(it)-1)}{j_\vth(it)-j_\vth(z)}t^{s-1} dt + \int_{t_0}^y \frac{j_\vth(z)}{j_\vth(it)-j_\vth(z)}t^{s-1} dt.
	\end{align}
	Rewriting
	\begin{equation*}
		\frac{j_\vth(z)}{j_\vth(it)-j_\vth(z)} = -1 + \frac{j_\vth(it)}{j_\vth(it)-j_\vth(z)}
	\end{equation*}
	in \eqref{eqn1} yields
	\begin{multline}\label{eqn:Fz1/yt0}
		F_{z,\frac1y,t_0}(s) = -\int_\frac1y^{t_0} \left(\vth(it)-\frac{1}{\sqrt t}\right)t^{s-1} dt + \int_\frac1y^{t_0} \frac{j_\vth(it)\left(\vth(it)-\frac{1}{\sqrt t}\right)}{j_\vth(it)-j_\vth(z)}t^{s-1} dt\\
		- \int_\frac1y^{t_0} t^{s-\frac32} dt + \int_\frac1y^{t_0} \frac{j_\vth(it)}{j_\vth(it)-j_\vth(z)}t^{s-\frac32} dt.
	\end{multline}
	The first and the third term contribute the first three terms in \eqref{eqn:Fz1/yt0final}.
	
	We next claim that
	\begin{equation}\label{eqn:Iz1/yt02}
		\lim_{y\to\infty} \int_\frac1y^{t_0} \frac{j_\vth(it)\left(\vartheta(it)-\frac{1}{\sqrt{t}}\right)}{j_\vth(it)-j_\vth(z)} t^{s-1} dt = 0.
	\end{equation}
	We split the integral into two parts, one from $\frac1y$ to $\frac{1}{\sqrt{y}}$ and one from $\frac{1}{\sqrt{y}}$ to $t_0$, and then use Lemma \ref{lem:ThetaGrowth} (2) and Lemma \ref{lem:jthetagrowth} (2). For the integral from $\frac{1}{\sqrt{y}}$ to $t_0$, Lemma \ref{lem:jthetagrowth} (2) implies that 
	\[
		\left|\frac{j_\vth(it)}{j_\vth(it)-j_\vth(z)}\right| \ll \frac{e^{\pi\sqrt{y}}}{e^{\pi y}-e^{\pi\sqrt{y}}} \ll e^{-\pi y+\pi \sqrt{y}}.
	\]
	Combining this with Lemma \ref{lem:ThetaGrowth} (2) yields, as $y\to\infty$,
	\begin{equation}\label{eqn:1/sqrtytot0}
		\left|\int_\frac{1}{\sqrt{y}}^{t_0} \frac{j_\vth(it)\left(\vartheta(it)-\frac{1}{\sqrt{t}}\right)}{j_\vth(it)-j_\vth(z)} t^{s-1} dt\right|\ll_{s,t_0} e^{-\pi y+\pi\sqrt{y}} \to 0.
	\end{equation}
	Lemma \ref{lem:jthetagrowth} (2) implies that for $\frac{1}{y}<t<t_0$
	\begin{equation}\label{eqn:jit/jit-jz}
		\frac{j_\vth(it)}{j_\vth(it)-j_\vth(z)}=-\frac{e^{\pi\left(\frac{1}{t}+i z\right)}}{1-e^{\pi\left(\frac{1}{t}+iz\right)}}\left(1+O\left(e^{-\frac{\pi}{t}}\right)\right).
	\end{equation}
	Lemma \ref{lem:ThetaGrowth} (2) and \eqref{eqn:jit/jit-jz} then gives that
	\begin{equation}\label{error}
		\frac{j_\vth(it)\left(\vartheta(it)-\frac{1}{\sqrt{t}}\right)}{j_\vth(it)-j_\vth(z)} t^{s-1}= O\left(\frac{e^{\pi\left(\frac{1}{t}+i z\right)}}{1-e^{\pi\left(\frac{1}{t}+iz\right)}}\left(1+O\left(e^{-\frac{\pi}{t}}\right)\right)t^{\sigma-\frac{3}{2}}e^{-\frac{\pi}{t}}\right).
	\end{equation}
	We then note that, for $t>\frac1y$ and $x\notin\Z$, we have
	\[
		\frac{1}{1-e^{\pi \left(\frac{1}{t}+iz\right)}} \ll_x 1, 
	\]
	Therefore \eqref{error} becomes $O_x(e^{-\pi y}t^{\sigma-\frac32})$. Hence, as
	\[
		\int_\frac1y^\frac{1}{\sqrt y} \frac{j_\vth(it)\left(\vth(it)-\frac{1}{\sqrt t}\right)}{j_\vth(it)-j_\vth(z)}t^{s-1} dt \ll_x e^{-\pi y}\int_\frac1y^\frac{1}{\sqrt y} t^{\sigma-\frac32} dt \to 0,
	\]
	where we use the fact that $\int_\frac1y^\frac{1}{\sqrt y}t^{\sigma-\frac32}dt$ grows at most polynomially in $y$. Combining this with \eqref{eqn:1/sqrtytot0} establishes \eqref{eqn:Iz1/yt02}.
	
	To evaluate the final term in \eqref{eqn:Fz1/yt0}, we plug in \eqref{eqn:jit/jit-jz} to rewrite it as
	\[
		-\lim_{\eps\to 0^+}\int_{\frac{1}{(1-\eps)y}}^{t_0} \frac{e^{\pi \left(\frac{1}{t}+ iz\right)}}{1-e^{\pi \left(\frac{1}{t}+iz\right)}}\left(1+O\left(e^{-\frac{\pi}{t}}\right)\right)t^{s-\frac{3}{2}} dt.
	\]
	Making the change of variables $t\mapsto\frac{1}{2t}$, this becomes 
	\begin{equation}\label{eqn:intBKgenL}
		-2^{\frac{1}{2}-s}\lim_{\eps\to0^+} \int_\frac{1}{2t_0}^{(1-\eps)\frac{y}{2}} \frac{e^{2\pi\left(t+\frac{iz}{2}\right)}}{1-e^{2\pi\left(t+\frac{iz}{2}\right)}} \left(1+O\left(e^{-2\pi t}\right)\right) t^{-s-\frac{1}{2}} dt.
	\end{equation}
	Up to the factor $-2^{\frac12-s}$ in front, this is precisely \cite[(4.13)]{BKgenL} with $s\mapsto s+\frac32$, $z\mapsto\frac z2$, and $t_0\mapsto2t_0$. If $\sigma+\frac32\ge1$ (i.e., $\sigma\ge-\frac12$), then we can follow \cite[(4.14)--(4.17)]{BKgenL} and use \cite[(2.4) with $x\mapsto-x$]{BKgenL} to obtain that the integral in \eqref{eqn:intBKgenL} equals the second and third terms in \cite[(4.6)]{BKgenL} with $s\mapsto s+\frac32$ and $z\mapsto\frac z2$, giving that for $\sigma\ge-\frac12$
	\[
		\lim_{\eps\to0^+} \int_\frac{1}{2t_0}^{(1-\eps)\frac{y}{2}} \frac{e^{2\pi\left(t+\frac{iz}{2}\right)}}{1-e^{2\pi\left(t+\frac{iz}{2}\right)}} \left(1+O\left(e^{-2\pi t}\right)\right) t^{-s-\frac12} dt = \frac{1}{2\pi}\Li_1\left(e^{\pi ix}\right) \left(\frac{y}{2}\right)^{-s-\frac12} + o_{x,s}(1).
	\]
	This completes the proof of \eqref{eqn:Fz1/yt0final}.
\end{proof}

We are now ready to prove Theorem \ref{thm:main}.

\begin{proof}[Proof of Theorem \ref{thm:main}]
	Parts (1) and (2) are given in Proposition \ref{prop:exist}. Taking $k=\frac12$ and 
	\begin{align}\label{eqn:Cells}
		C_{\ell,s}(x) &:=
		\begin{cases}
			\frac1s & \text{if }\ell=0,\\
			\frac1\pi\left((s+1-\ell)_{\ell-1}\Li_\ell\left(e^{-\pi ix}\right)-(1-s)_{\ell-1}\Li_\ell\left(e^{-\pi ix}\right)\right) & \text{if }\ell\in\N,
		\end{cases}\\
		\label{eqn:Dells}
		D_{\ell,s}(x) &:= C_{\ell,\frac12-s}(x),
	\end{align}
	the statement of part (3) is precisely the statement of Proposition \ref{prop:limit}.
\end{proof}

\end{document}